\documentclass[twoside]{aiml}

\usepackage{aimlmacro}

\usepackage{graphicx}
\usepackage{amsmath}
\usepackage{amssymb}
\usepackage{multicol}

\usepackage{tikz}

\DeclareMathOperator{\possible}{\text{\tikz[scale=.6ex/1cm,baseline=-.6ex,rotate=45,line width=.1ex]{
                            \draw (-1,-1) rectangle (1,1);}}}
\DeclareMathOperator{\necessary}{\text{\tikz[scale=.6ex/1cm,baseline=-.6ex,line width=.1ex]{
                            \draw (-1,-1) rectangle (1,1);}}}

\DeclareMathOperator{\dotnecessary}{\text{\tikz[scale=.6ex/1cm,baseline=-.6ex,line width=.1ex]{
                            \draw (-1,-1) rectangle (1,1);  \draw[fill=black] (0,0) circle (.25);}}}

\usepackage{prftree}

\usepackage{quiver}
\usepackage{commands}

\def\lastname{Sierra Miranda, Studer and Zenger}

\begin{document}

\begin{frontmatter}
  \title{Coalgebraic proof translations for non-wellfounded proofs}
  \author{Borja Sierra Miranda}\footnote{borja.sierra@unibe.ch, supported by the Swiss National Science Foundation (SNSF) under grant agreement Nr. 200021\_214820 (Non-wellfounded and cyclic proof theory)}
  \author{Thomas Studer}\footnote{thomas.studer@unibe.ch, supported by the Swiss National Science Foundation (SNSF) under grant agreement Nr. 200021\_214820 (Non-wellfounded and cyclic proof theory)}
  \author{Lukas Zenger}\footnote{lukas.zenger@unibe.ch, supported by the European Union's Horizon 2020 research and innovation program under the Marie Skłodowska-Curie grant agreement Nr. 101007627}
  \address{Institut f\"ur Informatik, Universit\"at Bern\\ Neubr\"uckstrasse 10, 3012 Bern, Switzerland}

  \begin{abstract}
    Non-wellfounded proof theory results from allowing proofs of infinite height in proof theory.
    To guarantee that there is no vicious infinite reasoning, it is usual to add a constraint
    to the possible infinite paths appearing in a proof. Among these conditions, one of the 
    simplest is enforcing that any infinite path goes through the premise of a rule infinitely often.
    Systems of this kind appear for modal logics with conversely well-founded frame conditions like \(\textsf{GL}\) or \(\Grz\).
  
    In this paper, we provide a uniform method to define proof translations for such systems, guaranteeing
    that the condition on infinite paths is preserved. In addition, as particular instance of our method,
    we establish cut-elimination for a non-wellfounded system of the logic \(\Grz\). 
    Our proof relies only on the categorical definition of corecursion via coalgebras, while an earlier proof
    by Savateev and Shamkanov uses ultrametric spaces and a corresponding fixed point theorem.
  \end{abstract}

 \end{frontmatter}

\section{Introduction}

In proof theory proofs are traditionally viewed as finite trees that are labelled by sets of formulas (called sequents) according to some fixed set of rules. In recent decades a new notion of proof has gained prominence, namely the notion of a non-wellfounded proof. These are proofs in the traditional sense --- i.e. labelled trees --- that allow some branches to be infinitely long, whence the name non-wellfounded. Such proofs can be considered to be a formal counterpart of proofs by infinite descent \cite{Brotherston-Simpson:07}. In order to ensure soundness in the presence of infinite branches, a global correctness condition is imposed that distinguishes correct infinitary reasoning from vicious reasoning. One kind of global soundness condition is to enforce that every infinite branch must pass infinitely often through a specific premise of some rule. We call this a \emph{local progress condition}. Non-wellfounded proof systems with a local progress condition have been developed for modal logics with conversely well-founded frame conditions like Gödel-Löb logic $\mathsf{GL}$ \cite{ShamkanovGl} or Grzegorczyk modal logic $\mathsf{Grz}$ \cite{ShamkanovGrz}. It is this kind of soundness condition that is studied in this paper.

Apart from $\mathsf{GL}$ and $\mathsf{Grz}$ there are many other logical systems that allow for non-wellfounded proofs. In particular, non-wellfounded proof systems have been developed for modal fixed point logics, such as the modal mu-calculus~(see e.g. \cite{Studer2008,Stirling2013,Niwinski1996}) as well as many of its fragments including the alternation free mu-calculus~\cite{Marti2021}, common knowledge~\cite{BUCHELI201083,Rooduijn22}, temporal logics~\cite{KokkinisStuder+2016+171+192,turata2023}, or program logics such as PDL~\cite{Docherty2019}. More recently non-wellfounded proof systems have been developed for intuitionistic modal fixed point logics \cite{Afshari2023,guillermo} and also for Peano arithmetic~\cite{Simpson2017}. However, proof systems for these logics usually require a stronger global constraint than the local progress condition, namely a so-called \emph{trace condition}. As trace conditions are not considered in this paper, we will not give more details. The interested reader may consult one of the papers cited above.

Proof translations are commonly studied in structural proof theory, one example being cut-elimination, where arbitrary proofs are translated into cut-free proofs of the same sequents. For finite proofs, cut-elimination is established by showing how to permute instances of cuts upwards, and then how to eliminate instances of cuts at the leaves. It is then shown that applying such permutations recursively yields a cut-free proof. For non-wellfounded proofs such a translation is aggravated due to the fact that the local permutations must preserve the global soundness criterion. Neverthless, results about proof translations for non-wellfounded proofs are found in the literature. To name a few, Savateev and Shamkanov establish cut-elimination for non-wellfounded proofs for $\mathsf{Grz}$~\cite{ShamkanovGrz} using ultrametric spaces and a fixed point theorem to guarantee that the limit of their construction is a (cut-free) proof. Baelde et al. \cite{baelde2016} and Saurin \cite{Saurin2023} consider the problem of cut-elimination for non-wellfounded proofs for linear logic with fixed points, and Das and Pous \cite{Das2018} establish cut-elimination for non-wellfounded proofs for the equational theory for Kleene algebra. Recently, Afshari and Kloibhofer proved cut-elimination for a cyclic calculus for modal logic with an `eventually' operator \cite{cut-cyclic}.

In this paper we study proof translations for non-wellfounded proofs based on a local progress condition. Our main contribution is the development of a uniform method to define such proof translations. Given two calculi $\mathcal{C}_0$ and $\mathcal{C}_1$ based on a local progress condition, we show how to define a map that translates proofs in $\mathcal{C}_0$ into proofs in $\mathcal{C}_1$. Depending on the considered translation, the map can be made to satisfy desirable properties, such as preserving the root sequent. This map is defined by first dissecting non-wellfounded proofs into finite fragments. It is then shown how the translation operates on a fragment locally, which we call a \emph{translation step}. Finally, the entire translation is defined from translation steps by corecursion. As corecursion is essentially a coalgebraic technique, our method is formulated within the realm of coalgebra and draws essentially from its concepts. In particular, the corecursive definition of the translation map is given by the unique morphism from a coalgebra into the final coalgebra of a given endofunctor. 

The developed method is uniform: it can be applied to all non-wellfounded proofs based on a local progress condition, independently of the specific logic, and to any kind of proof translation. We illustrate this by giving a concrete example: we establish cut-elimination for non-wellfounded proofs for $\mathsf{Grz}$. In difference to the proof of Savateev and Shamkanov \cite{ShamkanovGrz}, we do not require to take a detour over ultrametric spaces. Instead it suffices to show how to push instances of cut out of the main fragment of a non-wellfounded proof. Our method then implies the existence of a proof translation mapping non-wellfounded proofs with cuts onto cut-free proofs. This also illustrates that proof translations within our framework can be considered to be close in spirit with proof translations for finite proofs. As for finite proofs, it requires to show how to translate a finite fragment of the proof tree to obtain the entire translation.

We assume that the reader is familiar with basic concepts of coalgebra, such as the definitions of a coalgebra, an endofunctor, a coalgebra morphism and a final coalgebra. An introduction to these concepts can be found in \cite{VENEMA2007331}.

In Section 2 we introduce the mathematical machinery needed to develop our method. In particular we define the endofunctor that allows us to define proof translations by corecursion and so-called finite-fragmented trees, that formalize the intuition of dissecting proofs. In Section 3 we define proof translations and prove their correctness. Section 4 then consists of the cut-elimination proof for $\mathsf{Grz}$.

\section{Corecursion}

\subsection{Basic definitions}

We will refer to (possibly empty) finite sequences of \(\nat\) as \emph{words} and denote the collection of all words as \(\nat^*\). We will use \(w,v,u\) to denote words. The empty word is \(\epsilon\), \(\nat^+\) is the collection of non-empty words. Concatenation of words is denoted as juxtaposition, and so is adding a number to the start or the end of a word. The prefix order on words is denoted as \(\worder\). \(\disjoint{w}{v}\) means that \(w,v\) are disjoint\footnote{In words, that they do not have a common \(\worder\)-upperbound.}. If \(W \subseteq \nat^*\) then we define \(w \circ W = \{wv \mid v \in W\}\).

We will also work with finite sequences of \(\nat^+\), which are denoted by \(r,s\). The empty sequence of words is denoted as \(\nil\), and concatenation is again simply juxtaposition. \(w{:}r\) and \(r{:}w\) denote the sequences obtained by adding the word \(w\) to the start of \(r\) or to the end of \(r\) respectively. The prefix order on these sequences is denoted as \(\lorder\). The number of elements of \(r\) is denoted as \(|r|\).

If \(f\) is a function and \(A\) is a subset of its domain, we will write \(\restricts{f}{A}\) to denote the restriction of \(f\) to \(A\). Notice that if \(w\) is a finite (infinite) sequent, then it is a function with domain a natural number (with domain \(\mathbb{N}\)). Then \(\restricts{w}{n}\) will just be the word consisting in the first \(n\) letters of \(w\).

A \emph{tree} is an ordered pair consisting of a (finite or infinite) collection of words (called \emph{nodes}) closed by the prefix order and a labelling function. Trees are assumed to be finitely branching. Therefore each node \(w\) of a tree has an arity \(k \in \nat\) such that for any \(i \in \nat\), \(wi\) is a node iff \(i < k\). The set of nodes of a tree $\tau$ is denoted by $\nodes{\tau}$ and the set of \emph{leaves} (maximal words by the prefix order belonging to the nodes) by $\leaf{\tau}$. We fix an arbitrary set $A$ whose elements will be used as labels as well as an object $\ast$ not belonging to $A$.

\subsection{Endofunctor}
One way to look at a non-wellfounded tree is as a collection of finite trees with instructions on how to glue the root of each of these finite trees (except one which is at the root) to the leaves of other finite trees in the collection. In order to specify to which leaves of a finite tree  other finite trees can be glued to we introduce the following notion.

\begin{definition}
    A \emph{finite tree with non-wellfounded leaves} is a finite tree \(\iota\) with labels in \(A \cup \{*\}\) that satisfies the following two properties:
    \begin{align*}
       \textup{(i) }\ell^\iota(\epsilon) \neq *, && \textup{(ii) }\ell^\iota(w) = * \text{ implies }w \in \leaf{\iota}.
    \end{align*}
    We denote finite trees with non-wellfounded leaves by \(\iota\) and the collection of finite trees with non-wellfounded leaves by \(\nwtree\). Given \(\iota \in \nwtree\) we define the sets:
    \begin{align*}
       \nwleaf{\iota} = \{w \in \nodes{\iota} \mid \ell^\iota(w) = * \}, &&
       \pnodes{\iota} = \nodes{\iota}\setminus\nwleaf{\iota}.
    \end{align*}
    The elements of the first set are called \emph{non-wellfounded leaves of \(\iota\)} and the elements of the second set are called \emph{proper nodes of \(\iota\)}.
\end{definition}

We define the endofunctor which we need to obtain the desired definition by corecursion.
\begin{definition}
    We define an endofunctor \(\treefun\) over \(\textbf{Set}\) (the category of sets) as:
    \begin{enumerate}
        \item \(\mathcal{T}(X) := \{(\iota,\mu) \mid \iota \in \nwtree \text{ and } \mu : \nwleaf{\iota} \longrightarrow X\}\),
        \item \(\mathcal{T}(f: X \longrightarrow Y) = \Big((\iota,\mu) \mapsto (\iota,f \circ \mu) \Big)\).
    \end{enumerate}
\end{definition}
From now on coalgebra will mean \(\treefun\)-coalgebra and coalgebra morphism will mean \(\treefun\)-coalgebra morphism.

Given a coalgebra \((C,\alpha)\) and an element \(c \in C\) with \(\alpha(c) = (\iota,\mu)\) we will write \(\iota^\alpha_c\) to refer to \(\iota\) and \(\mu^\alpha_c\) to refer to \(\mu\). In case \(\alpha\) can be deduced from context we will omit it.
\begin{definition}
    Let \((C,\alpha)\) be a coalgebra, \(c \in C\) and \(r\) be a finite sequence of \(\nat^+\). We say that \(r\) is an \emph{\(\alpha\) root-path of \(c\)} iff
    \begin{align*}
        &\textsf{nil} \text{ is always a root-path of }c, \\
        & w {:} r \text{ is a root-path of }c \text{ iff }w \in \nwleaf{\iota_c} \text{ and }r \text{ root-path of }\mu^\alpha_c(w).
    \end{align*}
    We will usually denote root-paths with \(r\) or \(s\) and the collection of \(\alpha\)-root paths of \(c\) as \(\rootsp{\alpha, c}\).
\end{definition}

 The notion of root-path is useful to traverse a coalgebra \((C,\alpha)\). Let us have a \(c_0 \in C\) and \(r \in \rootsp{\alpha,c}\). The first word \(w\) of \(r\) is a non-wellfounded leaf of \(\iota^\alpha_c\). Using \(\mu^\alpha_c\) we can obtain a new element of the coalgebra \(c_1 = \mu^\alpha_{c_0}(w)\). We can think of \(c_1\) as the result of traversing \(c_0\) using the \(w\) branch. We can iterate this procedure to exhaust the full root-path. This means that given an element of a coalgebra and one of its root-paths we can obtain a new element of the coalgebra. This concept is formalized in the next definition.

\begin{definition}
    Let \((C,\alpha)\) be a coalgebra, \(c \in C\) and \(r \in \rootsp{\alpha, c}\). We define the \emph{\(\alpha\) subelement of \(c\) generated by \(r\)} by recursion in \(r\) as:
    \begin{align*}
        &\subelement{\alpha}{c}{\textsf{nil}} = c,
        &\subelement{\alpha}{c}{w {:} r} = \subelement{\alpha}{\mu^\alpha_c(w)}{r}.
    \end{align*}
\end{definition}

We can also talk about the tree with non-wellfounded leaves of a subelement, this is called a \emph{fragment} of the original element of the coalgebra.

\begin{definition}
    Let \((C,\alpha)\) be a coalgebra, \(c \in C\) and \(r \in \rootsp{\alpha, c}\). We define the \emph{\(\alpha\)-fragment of \(c\) at \(r\)} by recursion in \(r\) as:
    \begin{align*}
        &\gfragment{\alpha}{c}{\textsf{nil}} = \iota^\alpha_c,
        &\gfragment{\alpha}{c}{w{:}r} = \gfragment{\alpha}{\mu^\alpha_c(w)}{r}.
    \end{align*}
    The fragment \(\gfragment{\alpha}{c}{\textsf{nil}}\) is also called the \emph{main \(\alpha\)-fragment of \(c\)}.
\end{definition}

The naming of these concepts: root-paths, fragments and subelements, are motivated from the final coalgebra of the endofunctor \(\treefun\). The objects of the final coalgebra will be non-wellfounded trees expressed as a finite tree which has finite trees attached to some of its leaves, which may also have attached finite trees to some of their leaves, \ldots If we call roots the points where a finite tree start, a root-path is just a path of roots. A subelement will just be a subtree which do not break any of the finite trees in two and a fragment of a root-path is just a finite tree occurring at the end of the path of roots.

We state some basic properties of root-paths, fragments and subelements.

\begin{lemma}
    \label{lm:fundamental-properties}
    Let \((C,\alpha)\) be a coalgebra. We have that:
\begin{enumerate}
    \item $r \in \rootsp{\alpha, c}$ and \(s \in \rootsp{\alpha, \subelement{\alpha}{c}{r}}\) implies: \(rs \in \rootsp{\alpha, c}\), \(\gfragment{\alpha}{\subelement{\alpha}{c}{r}}{s} = \gfragment{\alpha}{c}{rs}\) and \(\subelement{\alpha}{\subelement{\alpha}{c}{r}}{s} = \subelement{\alpha}{c}{rs}\).
    \item \(r \in \rootsp{\alpha, c}\) implies that for any \(w \in \nat^+\), \(r{:}w \in \rootsp{\alpha, c}\) iff \(w \in \nwleaf{\gfragment{\alpha}{c}{r}}\).
    \item \(s \lorder r\) and \(r \in \rootsp{\alpha, c}\) implies that \(s \in \rootsp{\alpha, c}\).
    \item If \((C,\alpha),(D,\beta)\) are coalgebras and \(\gamma : C \longrightarrow D\) is a coalgebra morphism, we have that:
        \(r \in \rootsp{\alpha, c} \text{ iff }r \in \rootsp{\beta, \gamma(c)}\), \(\gfragment{\alpha}{c}{r} = \gfragment{\beta}{\gamma(c)}{r}\) and \(\gamma(\subelement{\alpha}{c}{r}) = \subelement{\beta}{\gamma(c)}{r}\).
    \end{enumerate}
\end{lemma}
\begin{proof}
    All of these are simple inductions in \(r\).
\end{proof}

\subsection{Finite-fragmented trees} \label{sec:fftree}

Finite-fragmented trees are non-wellfounded trees with a partition of the nodes. The idea is that each element of the partition gives a finite tree. This will imply that finite-fragmented trees with an appropiate function will form a final coalgebra of \(\treefun\).

\begin{definition}
    Let \(F \subseteq \nat^*\). We say that  $F$ fulfills the \emph{fragmentation properties} iff the following conditions hold:
    \begin{enumerate}
        \item (Finiteness) \(F\) is finite.
        \item (Root) There is a \(w \in F\) such that for any \(v \in F\), \(w \worder v\).
        \item (Convexity) For any \(w,v \in F\) if \(w \worder u \worder v\) then \(u \in F\).
    \end{enumerate}
    Given such a set \(F\), we will denote by \(\rt{}{}{F}\) its unique root, in other words its minimum element.
\end{definition}

\begin{definition}
    A finite-fragmented tree is a pair \(\pi = (\tau,\fragmentation{})\) consisting of:
    \begin{enumerate}
        \item A non-wellfounded tree \(\tau\).
        \item A partition \(\fragmentation{}\) of \(\nodes{\tau}\) such that any \(F \in \fragmentation{}\) fulfills the fragmentation properties, this \(\mathcal{F}\) is called \emph{the fragmentation of \(\pi\)}.
    \end{enumerate}
    We will usually denote finite-fragemented trees with \(\pi\) and the collection of all finite-fragmented trees as \(\fftree\). Given $\pi=(\tau, \fragmentation{})$ we also denote $\fragmentation{}$ by $\fragmentation{\pi}$ and $\tau$ by $\tau^\pi$. Given \(\pi \in \fftree\) we define the roots of \(\pi\) as the set \(\roots{\pi} = \{\rt{}{}{F} \mid F \in \fragmentation{\pi} \}\) and given a node \(w\) of \(\pi\), \({^{\pi}\sqrt{w}}\) will be the minimum of the equivalence class \(w\) belongs to in \(\pi\).

    We will write \((\tau,\sim)\) for the finite-fragmented tree whose fragmentation is the partition given by the equivalence relation \(\sim\).
\end{definition}

The roots of a finite-fragmented tree mark the end of a finite tree and the start of a new one. We define an immediate successor relation for roots and also assign a measure to nodes depending on the roots below them.

\begin{definition}
    For \(\pi \in \fftree\) define the relation \(\impred{\pi}{}{} \subseteq \roots{\pi}\times \roots{\pi}\) as:
    \[
    \impred{\pi}{w}{v} \text{ iff } w \worders v \text{ and there is no } u \in \roots{\pi} \text{ such that } w \worders u \worders v.
    \]
    Observe that if \(\impred{\pi}{w}{u}, \impred{\pi}{v}{u}\) then \(w = v\).
    Given \(w \in \nodes{\pi}\) define its \emph{\(\pi\)-fragmentation height} as 
    \[\fheight{\pi}{w} = |\{ v \worders w \mid v \in \roots{\pi} \}|,\] where \(|X|\) means the cardinality of \(X\). 
\end{definition}

Given a root \(w\) of \(\pi\) define the finite tree starting at \(w\) and the (finite-fragmented) subtree of $\pi$ starting at \(w\) as follows.

\begin{definition}
    Let \(\pi \in \fftree\) and \(w \in \roots{\pi}\). We define the \emph{tree fragment in \(\pi\) given by \(w\)}, denoted \(\fragment{\pi}{w}\), as the finite tree with non-wellfounded leaves \((\nodes{},\ell)\), where:
    \begin{align*}
        &\nodes{} = \{v \in \nat^*\mid w \sim^\pi wv\} \cup \{ u \in \nat^+\mid \impred{\pi}{w}{wu} \}, \\
        &\ell(v) = \begin{cases}
            \ell^\pi(wv) &\text{if }w \sim^\pi wv, \\
            * &\text{if }\impred{\pi}{w}{wv}.
        \end{cases} 
    \end{align*}
\end{definition}

\begin{definition}
    Let \(\pi \in \fftree\) and \(w \in \roots{\pi}\). We define the \emph{subtree of \(\pi\) generated by \(w\)}, denoted \(\subtree{\pi}{w}\), as the finite-fragmented tree \((\nodes{},\ell{},\sim)\) where:
    \begin{align*}
        &\nodes{} = \{ v \in \nat^* \mid wv \in \nodes{\pi}\}, \\
        &\ell(v) = \ell^\pi(wv), \\
        &v \sim u \text{ iff } wv \sim^\pi wu.
    \end{align*}
\end{definition}

Roots and tree fragments characterize finite-fragmented trees:

\begin{lemma}
    \label{lm:characterization-of-equality}
    Let \(\pi_0,\pi_1 \in \fftree\) such that \(\roots{\pi_0} = \roots{\pi_1}\) and for any \(w \in \roots{\pi_0} = \roots{\pi_1}\), \(\fragment{\pi_0}{w} = \fragment{\pi_1}{w}\). Then \(\pi_0 = \pi_1\).
\end{lemma}
\begin{proof}
    Let \(w \in \nodes{\pi_0}\), define \(v = \sqrt[\pi_0]{w}\) so \(w = vu\) for some \(u\) and \(u \in \pnodes{\fragment{\pi_0}{v}} = \pnodes{\fragment{\pi_1}{v}}\). By definition of tree fragment \(w = vu \in \nodes{\pi_1}\). So \(\nodes{\pi_0} \subseteq \nodes{\pi_1}\), the other direction is analogous so \(\nodes{\pi_0} = \nodes{\pi_1}\). Also, using the definition of tree fragment \(\ell^{\pi_0}(w) = \ell^{\fragment{\pi_0}{v}}(u) = \ell^{\fragment{\pi_1}{v}}(u) = \ell^{\pi_1}(w)\). All left to show is that \(w_0 \sim^{\pi_0} w_1\) iff \(w_0 \sim^{\pi_1} w_1\), we show the left to right direction the other being analogous.

    Let \(w_0 \sim^{\pi_0} w_1\) and define \(v = \sqrt[\pi_0]{w_0} = \sqrt[\pi_0]{w_1}\) so \(w_0 = v u_0\) and \(w_1 = v u_1\) for some \(u_0,u_1\). By definition of tree fragment \(u_0,u_1\in \pnodes{\fragment{\pi_0}{v}} = \pnodes{\fragment{\pi_1}{v}}\) so \(w_0 \sim^{\pi_1} v \sim^{\pi_1} w_1\), as desired.
\end{proof}

We state some basic properties of these constructions.

\begin{lemma}
    \label{lm:fundamental-properties-tree}
    Let \(\pi \in \fftree\) and 
    $w \in \roots{\pi}$. We have that:
        \begin{enumerate}
            \item  \(v \in \roots{\subtree{\pi}{w}}\) iff \(wv \in \roots{\pi}\). 
            \item \(\fragment{\subtree{\pi}{w}}{v} = \fragment{\pi}{wv}\).
            \item \(\subtree{\subtree{\pi}{w}}{v} = \subtree{\pi}{wv}\).
        \end{enumerate}
\end{lemma}
\begin{proof}
To show (i) we just note that \([v]_{\sim^{\subtree{\pi}{w}}} = \{ u \mid  wu \in [wv]_{\sim^\pi}\}\) (1). It is clear then that if \(wv\) is the minimum of \([wv]_{\sim^\pi}\) then \(v\) is the minimum of \([v]_{\sim^{\subtree{\pi}{w}}}\), which shows right to left. For the other direction let \(v \in \roots{\subtree{\pi}{w}}\), i.e. \(v\) is the minimum of \([v]_{\sim^{\subtree{\pi}{w}}}\) and let \(u \in [wv]_{\sim^\pi}\). By the root property there must be a \(u' = \sqrt[\pi]{[wv]_{\sim^\pi}}\) with \(u' \worder u\) and \(u' \worder wv\) (2). (2) implies that either \(w\worder u'\) or \(u' \worders w\), the second case being impossible since by convexity using that \(u' \worders w \worder wv\) we get \(w \in [u']_{\sim^\pi} = [wv]_{\sim^\pi}\) but \(w\) is a root so \(w \worder s\). Since \(w \worder u' \worder u\) there is \(u_0\) such that \(wu_0 = u\) and by (1) \(u_0 \in [v]_{\sim^{\fragment{\pi}{w}}}\) so \(v \worder u_0\) and then \(wv \worder wu_0 = u\), as desired. The equalities (ii) and (iii) are just by unfolding definitions.
\end{proof}

With these constructions we can put a coalgebra structure onto \(\fftree\).

\begin{definition}
    We define the function \(\destruct : \fftree \longrightarrow \treefun(\fftree)\) as:
    \[
    \pi \mapsto (\fragment{\pi}{\epsilon}, w \mapsto \subtree{\pi}{w})
    \]
    We will usually omit the \(\destruct\) when we write \(\rootsp{\destruct, \pi}\), \(\gfragment{\destruct}{\pi}{r}\), \(\subelement{\destruct}{\pi}{r}\), and so on.
\end{definition}
In words, given a finite fragmented tree we can destruct it into a finite fragment at its root, which is a tree with non-wellfounded leaves, and an attachment of finite-fragmented trees to the non-wellfounded leaves of the finite fragment.

Note that right now we have a duplication of concepts for finite-fragmented trees. We have roots, tree fragments and subtrees and from its coalgebraic structure with \(\destruct\) we also have root-paths, fragments and subelements. The following definition and the subsequent lemma make clear that these concepts are in fact analogous (which justifies the naming of the coalgebraic concepts).

\begin{definition}
    Let \((C,\alpha)\) be a coalgebra, \(c \in C\) and \(r \in \rootsp{\alpha, c}\). We define the \emph{word of $r$} recursively in $r$ as:
    \begin{align*}
        &\word{\nil} = \epsilon,
        &\word{w{:}r} = w \word{r}.
    \end{align*}
    We note that it is easy to show \(\word{r{:}w} = \word{r}w\).

    Given \(\pi \in \fftree\) and \(w \in \roots{\pi}\) it can be shown that there exists an unique sequence\footnote{For uniqueness just note that if \(\impred{\pi}{w}{u}\) and \(\impred{\pi}{v}{u}\) then \(w = v\)} \(w_0,\ldots,w_n \in \roots{\pi}\) such that \(\epsilon = \impred{\pi}{w_0}{\impred{\pi}{\cdots}{w_n}} = w\). We define \(\rootp{\pi}{w} = [w_1, w_2 - w_1, \ldots,w_n - w_{n-1}]\).
\end{definition}

\begin{lemma}
    \label{lm:correspondance-of-coalgebra-and-representation-operations}
    Let \(\pi \in \fftree\), we have that:
    \begin{align*}
        &\textup{(i) }\text{If \(r \in \rootsp{\pi}\), then:} & &\textup{(ii) }\text{If \(w \in \roots{\pi}\), then:}\\
        &\hspace{0.4cm}\textup{(a) }\word{r} \in \roots{\pi}, 
        &&\hspace{0.5cm}\textup{(a) }\rootp{\pi}{w} \in \rootsp{\pi},\\
        &\hspace{0.4cm}\textup{(b) }\rootp{\pi}{\word{r}} = r, 
        &&\hspace{0.5cm}\textup{(b) }\word{\rootp{\pi}{w}} = w,\\
        &\hspace{0.4cm}\textup{(c) }\gfragment{}{\pi}{r} = \fragment{\pi}{\word{r}}, 
        &&\hspace{0.5cm}\textup{(c) }\fragment{\pi}{w} = \gfragment{}{\pi}{\rootp{\pi}{w}},\\
        &\hspace{0.4cm}\textup{(d) }\subelement{}{\pi}{r} = \subtree{\pi}{\word{r}}.
        &&\hspace{0.5cm}\textup{(d) }\subtree{\pi}{w} = \subelement{}{\pi}{\rootp{\pi}{w}}.
    \end{align*}
\end{lemma}
\begin{proof}
    Proof of (i). We show the four statements simultaneously by induction in $|r|$. The \(\nil\) case is easy, so assume we have \(r{:}w\). First, since \(r{:}w \in \rootsp{}{\pi}\) we have that \(r \in \rootsp{}{\pi}\) and then we know that \(r{:} w \in \rootsp{}{\pi}\) implies that \(w \in \nwleaf{\gfragment{}{\pi}{r}} \overset{\text{I.H.}}{=} \nwleaf{\fragment{\pi}{\word{r}}}\). By definition of tree-fragment this means that \(\impred{\pi}{\word{r}}{\word{r}w}\) (1), so \(\word{r{:}w} = \word{r}w \in \roots{\pi}\).

    To show (b) we have to show that \(\rootp{\pi}{\word{r{:}w}} = r{:}w\), so let \(r = [w_0,\ldots,w_{n-1}]\) i.e. \(\word{r} = w_0\cdots w_{n-1}\). By I.H. we have that \(\rootp{\pi}{\word{r}} = r\) which implies that \(\impred{\pi}{\epsilon}{\impred{\pi}{w_0}{\impred{\pi}{w_0w_1}{\impred{\pi}{\cdots}{w_0\cdots w_{n-1}}}}}\). It suffices to show that \(w_0\cdots w_{n-1} = \impred{\pi}{\word{r}}{\word{r{:}w}} = w_0\cdots w_{n-1}w\). This is simply (1), showed above.

    To show (c) we just use the following reasoning:
    \begin{multline*}
    \gfragment{}{\pi}{r{:}w} \overset{(\text{2})}{=} \gfragment{}{\subelement{}{\pi}{r}}{[w]} \overset{\text{I.H.}}{=} \gfragment{}{\subtree{\pi}{\word{r}}}{[w]} \overset{(\text{3})}{=} \gfragment{}{\subtree{\subtree{\pi}{\word{r}}}{w}}{\nil} \\ \overset{(\text{4})}{=}
    \gfragment{}{\subtree{\pi}{\word{r}w}}{\nil} \overset{(\text{3})}{=} \fragment{\subtree{\pi}{\word{r}w}}{\epsilon}
    \overset{(\text{4})}{=} \fragment{\pi}{\word{r}w} = \fragment{\pi}{\word{r{:}w}},
    \end{multline*}
    where (2) is thanks to Lemma \ref{lm:fundamental-properties}, (3) are just unfolding of definitions and (4) are thanks to Lemma \ref{lm:fundamental-properties-tree}.

    Finally (d) is analogous to (c):
    \begin{multline*}
        \subelement{}{\pi}{r{:}w} \overset{(\text{5})}= \subelement{}{\subelement{}{\pi}{r}}{[w]} \overset{\text{I.H.}}{=} \subelement{}{\subtree{\pi}{\word{r}}}{[w]} \overset{(\text{6})}{=} \subtree{\subtree{\pi}{\word{r}}}{w} \\
        \overset{(\text{7})}{=} \subtree{\pi}{\word{r}w} = \subtree{\pi}{r{:}w}.
    \end{multline*}
    where (5) is thanks to Lemma \ref{lm:fundamental-properties}, (6) are just unfolding of definitions and (7) are thanks to Lemma \ref{lm:fundamental-properties-tree}.

    Proof of (ii). We show the four statements simultaneously by induction in $\fheight{\pi}{w}$. If \(\fheight{\pi}{w} = 0\) then \(w = \epsilon\) and the proof is easy, let us assumme that \(\fheight{\pi}{w} = n+1\), so there is \(v \in \roots{\pi}\) such that \(\fheight{\pi}{v} = n\) and \(\impred{\pi}{v}{w}\). Let \(u\) be such that \(w = vu\), we note that \(u \in \nat^+\) since \(v \worders w\). By I.H. \(\rootp{\pi}{v} \in \rootsp{}{\pi}\) and by definiton of fragment tree \(u \in \fragment{\pi}{v} \overset{\text{I.H}}{=} \gfragment{}{\pi}{\rootp{\pi}{v}}\), so by Lemma \ref{lm:fundamental-properties} we get \(\rootp{\pi}{v} {:} u \in \rootsp{}{\pi}\). But \(\rootp{\pi}{v} {:} u = \rootp{\pi}{w}\) thanks to \(\impred{\pi}{v}{w}\).

    To show (b) we remember that \(\rootp{\pi}{w} = \rootp{\pi}{v} {:} u\) and \(\word{\rootp{\pi}{v}} \overset{\text{I.H.}}{=} v\). Then \(\word{\rootp{\pi}{w}} = \word{\rootp{\pi}{v}{:} u} = \word{\rootp{\pi}{v}}u = vu = w\), as desired. Finally, (c) and (d) are easy to derive using that we already shown (i), (ii)(a) and (ii)(b).
\end{proof}

\subsection{Finite-fragmented trees as a final coalgebra}

We establish the desired result that \((\fftree,\destruct)\) is a final coalgebra of \(\treefun\). This will allow us to define functions to \(\fftree\) by corecursion. First, we state two technical lemmas.

\begin{lemma}
    \label{lm:about-root-paths}
    Let \((C,\alpha)\) be a coalgebra, \(c \in C\), \(r,s \in \rootsp{\alpha, c}\). Then:
    \begin{enumerate}
        \item \(|r| = |s|\) and \(r \neq s\) implies \(\disjoint{\word{r}}{\word{s}}\).
        \item \(|r| < |s|\) then either \(\disjoint{\word{r}}{\word{s}}\) or \(r \lorders s\).
        \item Then \(\word{r} \worder \word{s}\) implies \(r \lorder s\).
    \end{enumerate}
\end{lemma}
\begin{proof}
    (i) is proved by induction in \(|r|\) using Lemma \ref{lm:fundamental-properties} and that two distinct leaves of the same finite tree must be disjoint. (ii) is proved by induction on the difference \(|s| - |r|\) using (i). (iii) is a consequence of (i) and (ii) by cases in \(|r| = |s|, |r| < |s|\) or \(|s| < |r|\).
\end{proof}

\begin{lemma}
    \label{lm:disjointness}
    Let \((C,\alpha)\) be a coalgebra and \(c \in C\). For any \(r,s \in \rootsp{\alpha, c}\) distinct, we have that: \((\word{r} \circ \pnodes{\gfragment{\alpha}{c}{r}}) \cap (\word{s} \circ \pnodes{\gfragment{\alpha}{c}{s}})= \varnothing\).
\end{lemma}
\begin{proof}
    We can assume without loss of generaltity that $|r| \leq |s|$ and let \(v\) belong to the intersection, in particular \(v = \word{r}u_0 = \word{s}u_1\) for some \(u_0 \in \pnodes{\gfragment{\alpha}{c}{r}}\) and \(u_1 \in \pnodes{\gfragment{\alpha}{c}{s}}\). By Lemma \ref{lm:about-root-paths}, either \( \disjoint{\word{r}}{\word{s}}\) or \(r \lorders s\). The first case is impossible since \(v\) is an upperbound of both, so we must have \(r \lorders s\) and then \(s = r[w,\ldots]\) with \(w \worder u_0\). Since \(s\) is a root-path an \(r{:}w \lorder s\) we get that \(r{:}w\) is a root-path which implies that \(w \in \nwleaf{\gfragment{\alpha}{c}{r}}\). Then \(u_0 \in \pnodes{\gfragment{\alpha}{c}{r}}\) and \(w \worder u_0\) implies that \(w \worders u_0\), but \(u_0\) is a node of \(\gfragment{\alpha}{c}{r}\) and \(w\) is a leaf of the same tree, so \(w \worders u_0\) is impossible.
\end{proof}

Finally, we can prove the main theorem of this section.\footnote{Category-theory oriented readers may wonder why we take the burden to prove this. A more categorical approach would have proven the existence of a final coalgebra for \(\mathcal{T}\) and define the finite-fragmented trees as elements of this coalgebra. However, for our purposes of providing translations between proof systems we need to transform (some) non-wellfounded trees to finite-fragmented trees and back. If we take this categorical approach providing this transformation will be equivalent to directly proving that our definition of finite-fragmented tree provides a final coalgebra of \(\mathcal{T}\).}

\begin{theorem}
    \label{th:final-coalgebra}
    \((\fftree,\destruct)\) is a final coalgebra of \(\treefun\). 
\end{theorem}
\begin{proof}
    Existence.
    Let \((C,\alpha)\) be a coalgebra. We define \(f : C \longrightarrow \fftree\) such that given \(c \in C\) \(f(c)\) is the finite-fragmented tree \(\pi = (\nodes{},\ell,\fragmentation{})\) where:
    \begin{align*}
        &\nodes{} = \bigcup_{r \in \rootsp{\alpha, c}} \word{r} \circ \pnodes{\gfragment{\alpha}{c}{r}}, \\
        &\ell(w) = \ell^{\gfragment{\alpha}{c}{r}}(u) \text{ where \(r\) is the unique element of \(\rootsp{\alpha, c}\)} 
        \text{ and }\\ 
        &\hspace{1cm}\text{\(u\) is the unique element of  \(\pnodes{\gfragment{\alpha}{c}{r}}\) such that \(v = \word{r}u\),}\\
        &\fragmentation{} = \{ \word{r} \circ \pnodes{\gfragment{\alpha}{c}{r}} \mid r \in \rootsp{\alpha, c}\}.
    \end{align*}
    The labelling is well-defined (i.e. the things asserted to be unique are indeed unique) due to Lemma \ref{lm:disjointness}. Note that $\pi$ is indeed a finite-fragmented tree.

    We need to prove that \(f\) is a coalgebra morphism, i.e.\ we have to show that \(\iota^\alpha_c = \fragment{\pi}{\epsilon}\) and \(f \circ \mu^\alpha_c = (w \in \fragment{\pi}{\epsilon} \mapsto \subtree{\pi}{w})\).
    
    Proof of \(\iota^\alpha_c = \fragment{\pi}{\epsilon}\). By definition the proper nodes of \(\fragment{\pi}{\epsilon}\) is the equivalence class of epsilon in \(\pi\), which in this case is \(\epsilon \circ \pnodes{\iota^\alpha_c}\), with the same labels as in \(\iota^\alpha_c\), The non-wellfounded leaves are the \(w\) such that \(\impred{\pi}{\epsilon}{w}\). To show the desired equality all we need to show is that for any \(w \in \nodes{\pi}\):
    \[
    \impred{\pi}{\epsilon}{w} \text{ if and only if }w \in \nwleaf{\iota^\alpha_c}
    = \nwleaf{\gfragment{\alpha}{c}{\nil}}.
    \]

    Left to right. Since \(w \in \nodes{\pi}\) there must be a \(r \in \rootsp{\alpha}{c}\) such that \(w = \word{r}u\) where \(u \in \pnodes{\gfragment{\alpha}{c}{r}}\). But \(w \in \roots{\pi}\) means that it must be the minimum of its equivalecne class which is \(\word{r} \circ \pnodes{\gfragment{\alpha}{c}{r}}\), since \(\epsilon \in \pnodes{\gfragment{\alpha}{c}{r}}\) we can conclude that \(u = \epsilon\), i.e. \(\word{r} = w\) (which implies that \(|r| > 0\) since \(\epsilon \worders w\)). If \(|r| > 1\) then \(r = [v,\ldots]\) so \([v] \in \rootsp{\alpha}{c}\) and then \(v \in \roots{\pi}\) and \(\epsilon \worders v \worders w\), contradiction. So \(|r| = 1\) and since \(\word{r} = w\) we must have \(r = [w]\), which implies (Lemma \ref{lm:fundamental-properties}) that \(w \in \nwleaf{\gfragment{\alpha}{c}{\nil}}\).

    Right to left. Let \(w \in \nwleaf{\gfragment{\alpha}{c}{\nil}}\), since \(\nil \in \rootsp{\alpha}{c}\) by Lemma \ref{lm:fundamental-properties} we get that \([w] \in \rootsp{\alpha}{c}\). It is easy see then that by definition of \(\pi\), \(\word{[w]} = w \in \roots{\pi}\). We want to show that \(\impred{\pi}{\epsilon}{w}\), assume the contrary. Then there must be a \(v \in \roots{\pi}\) such that \(\epsilon \worders v \worders w\), but then there must be (with a reasoning analogous to the one in the left to right direction to show that \(w = \word{r}\)) a \(s \in \rootsp{\alpha}{c}\) such that \(\word{s} = v\). Since \(\word{s} = v \worders w = \word{[w]}\) by Lemma \ref{lm:about-root-paths} it must be the case that \(s \lorders [w]\), which means that \(s = \nil\). But this is absurd since \(\epsilon = \word{\nil} = \word{s} = v > \epsilon\).

    Let \(w \in \nwleaf{\iota^\alpha_c} = \nwleaf{\fragment{\pi}{\epsilon}}\), we have to show that \(f(\mu^\alpha_c(w)) = \subtree{\pi}{w}\). We note that by definition we have that:
    \[
    r \in \rootsp{\alpha}{\mu^\alpha_c(w)} \text{ iff } w{:}r \in \rootsp{\alpha}{c},
    \]
    \[
    \gfragment{\alpha}{\mu^\alpha_c(w)}{r} = \gfragment{\alpha}{c}{w{:}r}.
    \]
    Then 
    \[
    \nodes{f(\mu^\alpha_c(w))} = \bigcup_{r \in \rootsp{\alpha}{\mu^\alpha_c(w)}} \word{r} \circ\pnodes{\gfragment{\alpha}{\mu^\alpha_c(w)}{r}} =
    \bigcup_{r \in \{r \mid w{:}r \in \rootsp{\alpha}{c}\}}  \word{r} \circ \pnodes{\gfragment{\alpha}{c}{w{:}r}}. 
    \]
    To show that \(\subtree{\pi}{w}\) and \(f(\mu^\alpha_c(w))\) have the same nodes it suffices to prove that:
    \[
    \tag{i}
    \nodes{\subtree{\pi}{w}} =  \bigcup_{r \in \{r \mid w{:}r \in \rootsp{\alpha}{c}\}}  \word{r} \circ \pnodes{\gfragment{\alpha}{c}{w{:}r}} .
    \]
    The right to left inclusion is trivial since it is clear that if \(v\) belongs to the RHS set then \(wv\) belongs to the nodes of \(\pi\) and then \(v\) belongs to the LHS set. Assume that \(v \in \nodes{\subtree{\pi}{w}}\), i.e. \(wv \in \nodes{\pi}\) so there is an \(s \in \rootsp{\alpha}{c}\) and a \(u \in \pnodes{\gfragment{\alpha}{c}{s}}\)  such that \(wv = \word{s}u\). Since \(w \in \nwleaf{\iota^\alpha_c}\) then \([w] \in \rootsp{\alpha}{c}\) and since \(w \worder wv = \word{s}u\) we have that either \(w \worder \word{s}\) or \(\word{s} \worders w\). In the first case, by Lemma \ref{lm:about-root-paths} we have that \([w] \lorder s\) which implies that \(s = {w}{:}s'\) and then \(\word{s'}u = v\) is in the RHS set. If \(\word{s} \worders w\) then \(s \lorders [w]\) so \(s = \nil\) which implies that \(wv = \word{s}u = \epsilon u = u\) and that \(u \in \pnodes{\gfragment{\alpha}{c}{s}} = \pnodes{\gfragment{\alpha}{c}{\nil}} = \pnodes{\iota^\alpha_c}\). This is a contradiction since \(wv = u\) implies \(w \worder u\) but \(w\) is a non-wellfounded leaf of \(\iota^\alpha_c\) and \(u\) is not (so it cannot be equal and it cannot be bigger since \(w\) is a leaf).

    To show that \(f(\mu^\alpha_c(w))\) has the same labelling as \(\subtree{\pi}{w}\) we let \(v \in \nodes{\subtree{\pi}{w}}\), then (by (i)) there are \(w{:}r \in \rootsp{\alpha}{c}\) and \(u \in \pnodes{\gfragment{\alpha}{c}{w{:}r}}\) such that  \(wv = \word{w{:}r}u\). The following equalities straightforwardly follow from this:
    \[
    \ell^{f(\mu^\alpha_c(w))}(v) = \ell^{\gfragment{\alpha}{\mu^\alpha_c(w)}{r}}(u) = \ell^{\gfragment{\alpha}{c}{w{:}r}}(u) = \ell^\pi(wv) = \ell^{\subtree{\pi}{w}}(v).
    \]

    Finally, we show that \(f(\mu^\alpha_c(w))\) has the same fragmentation as \(\subtree{\pi}{w}\). Note that \(v \sim^{\subtree{\pi}{w}} u\) iff \(wv \sim^\pi wu\) iff\footnote{In principle it would be equivalent to the existence of some \(r \in \rootsp{\alpha}{c}\) and \(v_0,u_0\) such that \(\word{r}v_0 = wv\) and \(\word{r}u_0 = wu\). This equalities with the fact that \(w \in \nwleaf{\iota^\alpha_c}\) so \([w] \in \rootsp{\alpha}{c}\) allows us to derive the desired equivalence.} there are \(r \) such that \(w{:}r \in \rootsp{\alpha}{c}\) and \(v_0,u_0 \in \pnodes{\gfragment{\alpha}{c}{w{:}r}}\) such that \(\word{w{:}r}v_0 = wv\) and \(\word{w{:}r}u_0 = wu\). This is equivalent to the existence of \(r \in \rootsp{\alpha}{\mu^\alpha_c(w)}\) and \(v_0,u_0 \in \pnodes{\gfragment{\alpha}{\mu^\alpha_c(w)}{r}}\) such that \(\word{r}v_0 = v\) and \(\word{r}u_0 = u\), which is equivalent to \(v \sim^{f(\mu^\alpha_c(w))} u\).
    
    Uniqueness. Let \((C,\alpha)\) be a coalgebra and \(\gamma_0,\gamma_1 : C \longrightarrow \fftree\) be coalgebra morphisms and let \(\pi_0 = \gamma_0(c), \pi_1 = \gamma_1(c)\) with \(c \in C\). By Lemma \ref{lm:characterization-of-equality} it suffices to show that \(\roots{\pi_0} = \roots{\pi_1}\) (i) and that for any \(w \in \roots{\pi_0}\), \(\fragment{\pi_0}{w} = \fragment{\pi_1}{w}\) (ii).

    Proof of (i). We show \(\subseteq\), the other inclusion being analogous. If \(w \in \roots{\pi_0}\) then \(\rootp{\pi}{w} \in \rootsp{\pi_0}\) by Lemma \ref{lm:correspondance-of-coalgebra-and-representation-operations}. By Lemma \ref{lm:fundamental-properties} we have that 
    \[
    \rootsp{\pi_0} = \rootsp{\gamma_0(c)} = \rootsp{\alpha, c} = \rootsp{\gamma_1(c)} = \rootsp{\pi_1},
    \]
    by virtue of \(\gamma_0,\gamma_1\) being coalgebra morphisms. So \(\rootp{\pi_0}{w} \in \rootsp{\pi_1}\) and thanks to Lemma \ref{lm:correspondance-of-coalgebra-and-representation-operations} we get \(w = \word{\rootp{\pi_0}{w}} \in \roots{\pi_1} \).
    
    Proof of (ii). Since we showed that \(\roots{\pi_0} = \roots{\pi_1}\) we have that for any \(w \in \roots{\pi_0}\), \(\rootp{\pi_0}{w} = \rootp{\pi_1}{w}\). Fix \(w \in \roots{\pi_0}\) and define \(r = \rootp{\pi_0}{w} = \rootp{\pi_1}{w}\). Using Lemma \ref{lm:fundamental-properties} and Lemma \ref{lm:correspondance-of-coalgebra-and-representation-operations} we get the desired
    \[
    \fragment{\pi_0}{w} = \gfragment{}{\pi_0}{r} = \gfragment{}{\gamma_0(c)}{r} = \gfragment{\alpha}{c}{r} = \gfragment{}{\gamma_1(c)}{r} = \gfragment{}{\pi_1}{r} = \fragment{\pi_1}{w}.
    \]
\end{proof}

Observe that since \((\fftree,\destruct)\) is a final coalgebra, \(\destruct\) is a coalgebra isomorphism and thus has an inverse \(\construct\). It is easy to see that given a pair \((\iota,\mu : \nwleaf{\iota} \longrightarrow \fftree)\), \(\construct(\iota,\mu)\) is just the result of ``gluing'' \(\mu(w)\) at \(w\) in \(\iota\) for each \(w \in \nwleaf{\iota}\). Given a coalgebra \((C,\alpha)\) let us denote by \(\final{\alpha}\) the unique coalgebra morphism from it to \((\fftree, \construct)\). We have that:
\[
\final{\alpha} = \construct \circ \treefun(\final{\alpha}) \circ \alpha.
\]
This equation is what provides corecursion of its character. \(\alpha\) can be considered as performing 1 step of the corecursion and \(\final{\alpha}\) will be applying the whole corecursion. Then, the equation says that applying the whole corecursion is the same as applying 1 step obtaining a pair consisting of a finite tree and a function giving an element of the coalgebra for each leaf. Then apply the whole corecursion to the elements of the coalgebras obtaining finite-fragmented trees and glue these finite-fragmented trees to the finite tree given by \(\alpha\).

\section{Proof translations}
\label{sec:proof-trans}

In this section we provide sufficient conditions to ensure that a function defined by corecursion transforms proofs in a local progress calculus to proofs in another local progress calculus. We start by defining exactly what we mean by local progress calculi.
During this section we fix an arbitrary set \(\sequentset\) whose elements are called sequents. 

\begin{definition}
    A \emph{rule instance} consists in a pair \(P \Longrightarrow C\) where:
    \begin{enumerate}
        \item \(P\) is a finite sequence of sequents, the \emph{premises} of the instance.
        \item \(C\) is a sequent, the \emph{conclusion} of the instance.
    \end{enumerate}
    A \emph{rule} \(R\) is a non-empty set of rule instances. 
    A \emph{local-progress sequent calculus} is a pair \(\mathcal{C} = (\mathcal{R},L)\) where:
    \begin{enumerate}
        \item \(\mathcal{R}\) is a set of rules, called the \emph{rules of \(\mathcal{C}\)}.
        \item \(L\) is a function that given a rule \(R \in \mathcal{R}\) and an instance of the rule \((S_0,\ldots,S_{k-1} \Longrightarrow S) \in R\) it returns a subset of \(\{0,\ldots,k-1\}\). This subset is the set of premises of the rule instance that make progress. The function is called the \emph{local progress function of \(\mathcal{C}\)}. We will denote the application of \(L\) to a rule \(R\) and rule instance \(r \in R\) as \(L_R(r)\).
    \end{enumerate}
    We will use $\mathcal{C}$ to denote local progress sequent calculi.
\end{definition}

From now on we assume that all trees considered are labelled with ordered pairs \((S,R)\) where \(S\) is a sequent and \(R\) a rule. We denote the collection of non-wellfounded labelled trees by $\illtree$. Given a node \(w\) in a tree \(\tau\) we will write \(\sequent{\tau}{w}\) to mean the first component of its label and \(\seqrule{\tau}{w}\) to mean the second component of its label. We define the usual notion of non-wellfounded proof in the following definition.

\begin{definition}
    Let \(\pi \in \illtree\). We say that \(\pi\) is a \emph{pre-proof in \(\lpcalculus\)} iff for every node \(w\) of \(\pi\) we have that: \((S_0,\ldots,S_{k-1} \Longrightarrow S) \in R\),
    where \(w\) is \(k\)-ary in \(\pi\), \(S_i = \sequent{\pi}{wi}\) for \(i < k\), \(S = \sequent{\pi}{w}\) and \(R = \seqrule{\pi}{w}\).
    We denote the collection of pre-proofs of \(\mathcal{C}\) as \(\illpreproof{\mathcal{C}}\). 

    Let \(\tau \in \illpreproof{\mathcal{C}}\) and \(w \in \nodes{\tau}\). We say that \emph{\(w\) is progressing (in \(\tau\) with respect to \(\mathcal{C}\))} if there is \(v \in \nat^*, i \in \nat\) such that \(w = vi\) and \(i \in L_R(S_0,\ldots,S_{k-1} \Longrightarrow S)\) where \(v\) is \(k\)-ary in \(\tau\), \(S_j = \sequent{\tau}{vj}\) for \(j < k\), \(S = \sequent{\tau}{v}\) and \(R = \seqrule{\tau}{v}\).

    We say that \(\pi\) is a \emph{proof in \(\lpcalculus\)} iff it is a pre-proof and for any infinite branch \(b\) of \(\pi\) there are infinitely many \(i\)'s such that \(\restricts{b}{i}\) is progressing.
    The collection of proofs of \(\mathcal{C}\) is denoted as \(\illproof{\mathcal{C}}\).
\end{definition}

In the following two definitions, we define a notion of proof for finite-fragmented trees.

\begin{definition}
    Let \(\iota \in \nwtree\) and \(s : \nwleaf{\iota} \longrightarrow \sequentset\). We say that \((\iota,s)\) is a proof fragment in \(\mathcal{C}\) iff for any \(w \in \pnodes{\iota}\) of arity \(k\) in \(\iota\), if we define
    \begin{align*}
        &S_i = \begin{cases}
            \sequent{\iota}{wi} & \text{for }wi \not\in\nwleaf{\iota}, \\
            s(wi) &\text{for }wi \in \nwleaf{\iota},
        \end{cases} \text{ for }i < k \\
        &S = \sequent{\iota}{w}, \\
        &R = \seqrule{\iota}{w},
    \end{align*}
    then we have that:
    \begin{enumerate}
        \item \((S_0,\ldots,S_{k-1} \Longrightarrow S) \in R\).
        \item (Progress) For any \(i < k\), \(wi \in \nwleaf{\iota}\) iff \(i \in L_R(S_0,\ldots,S_{k-1}\Longrightarrow S)\).
    \end{enumerate}
    
\end{definition}

\begin{definition}
    Let \(\pi = (\tau,\sim) \in \fftree\). We say that \(\pi\) is a \emph{finite-fragmented proof in \(\mathcal{C}\)} iff for any \(r \in \rootsp{\pi}\)\footnote{Recall that we write $\rootsp{\pi}$ for $\rootsp{\destruct,\pi}$.} we have that the pair
    \[(\gfragment{}{\pi}{r}, \, \, v \in \nwleaf{\gfragment{}{\pi}{r}} \mapsto \sequent{(\gfragment{}{\pi}{r{:}v})}{\epsilon})\]
    is a proof fragment in \(\mathcal{C}\). We will denote the set of finite-fragmented proofs in \(\mathcal{C}\) as \(\ffproof{\mathcal{C}}\).
\end{definition}
    We note that in the previous definition \(\sequent{(\gfragment{}{\pi}{r{:}v})}{\epsilon} = \sequent{(\fragment{\pi}{\word{r}v})}{\epsilon} = \sequent{\pi}{\word{r}v}\), so it is just the sequent corresponding to the node in the whole proof.
    Given a preproof of \(\mathcal{C}\) we can define a partition as in the next definition.

\begin{definition}
    \label{df:fragment-relation-in-proofs}
    Let \(\tau \in \illpreproof{\mathcal{C}}\), we define the relation \(\noprogone{\tau} \subseteq \nodes{\tau} \times \nodes{\tau}\) as:
    \[
    w \noprogone{\tau} v \text{ iff there is \(i \in \nat\) such that }v = wi \text{ and }v\text{ is not progressing}.
    \]
    Let \(\noprog{\tau}\) be the equivalence relation closure of \(\noprogone{\tau}\), in other words \(w \noprog{\tau} v\) iff there is a sequence \(w_0,\ldots,w_n\) such that \(w = w_0, v = w_n\) and for any \(i < n\), \(w_i \noprogone{\tau} w_{i+1}\) or \(w_{i+1} \noprogone{\tau} w_i\). 
\end{definition}

 A relation between finite-fragmented proofs and non-wellfounded proofs is established and proven in the following lemma.

\begin{lemma}
    \label{lm:equivalent-proof-notions}
    We have that
    \begin{enumerate}
        \item \(\tau \in \illproof{\mathcal{C}}\) implies \((\tau,\noprog{\tau}) \in \ffproof{\mathcal{C}}\).\footnote{In fact it can be shown that the fragmentation in a finite-fragmented proof is unique. This point implies that \(\noprog{\tau}\) is that unique fragmentation.}
        \item \(\pi = (\tau,\sim) \in \ffproof{\mathcal{C}}\) implies \(\tau \in \illproof{\mathcal{C}}\).
    \end{enumerate}
\end{lemma}
\begin{proof}
To show that every node is an instance of a rule is easy in both implication, so we will focus in proving that the progressing condition is fulfilled in both implications.

Proof of (i). First we need to show that if \(\tau\) is a proof then \(\noprog{\tau}\) is a fragmentation. Convexity is shown by using that any path in a tree between \(w\) and \(v \geq w\) must go through all the middle nodes, i.e. through \(\{u \mid w \worder u \worder v\}\). The existence of root is shown by using that any path between \(w\) and \(v\) in a tree must go through \(w \cap v\), i.e. the common part of \(w\) and \(v\). Finiteness is shown using K\"onig's Lemma together with convexity and the existence of root for any equivalence class.

That the progress condition of proof-fragment is fulfilled follows from the definition of \(\noprog{\tau}\).\footnote{We also need to use that in a tree any path from \(w\) to \(wi\) must contain a transition \((w,wi)\).}

Proof of (ii). If we have an infinite branch in \(\tau\) it must go through infinitely many fragments, since each fragment is finite. By the second condition of proof fragment and thanks to each fragment being a proof fragment this implies that the branch makes progress infinitely many times.
\end{proof}

With the notion of finite-fragmented proof at hand we are ready to define the sufficient conditions to obtain a proof translation by corecursion.

\begin{definition}
    Let \(\mathcal{C}_0\) and \(\mathcal{C}_1\) be local progress sequent calculi and let
    \mbox{\(\alpha : \fftree \longrightarrow \treefun(\fftree)\)}. We say that \emph{\(\alpha\) is a proof translation step from \(\mathcal{C}_0\) to \(\mathcal{C}_1\)} iff we have that for any \(\pi \in \ffproof{\mathcal{C}_0}\):
    \begin{enumerate}
        \item \((\gfragment{\alpha}{\pi}{\nil}, w \mapsto \sequent{\gfragment{\alpha}{\pi}{[w]}}{\epsilon}) \) is a proof fragment in \(\mathcal{C}_1\).
        \item  For any \(w \in \nwleaf{\gfragment{\alpha}{\pi}{\nil}}\) we have that \(\subelement{\alpha}{\pi}{[w]} \in \ffproof{\mathcal{C}_0}\).
    \end{enumerate}

    A function \(\gamma\) is a \emph{finite-fragmented proof translation from \(\mathcal{C}_0\) to \(\mathcal{C}_1\)} iff 
    \mbox{\(\gamma : \ffproof{\mathcal{C}_0} \longrightarrow \ffproof{\mathcal{C}_1}\)}. A function \(\gamma\) is a \emph{non-wellfounded proof translation from \(\mathcal{C}_0\) to \(\mathcal{C}_1\)} iff \(\gamma : \illproof{\mathcal{C}_0} \longrightarrow \illproof{\mathcal{C}_1}\).
\end{definition}

\begin{theorem}
    \label{th:proof-translation-step-to-proof-translation}
    If \(\alpha\) is a proof translation step from \(\mathcal{C}_0\) to \(\mathcal{C}_1\) then 
    \begin{enumerate}
        \item \(\restricts{\final{\alpha}}{\ffproof{\mathcal{C}_0}}\), i.e. the function defined by corecursion from \(\alpha\) applied only to finite-fragmented proofs of \(\mathcal{C}_0\), is a finite-fragmented proof translation from \(\mathcal{C}_0\) to \(\mathcal{C}_1\), and
        \item The function \(\tau \in \illproof{\mathcal{C}_0} \mapsto (\mathsf{p}_0 \circ \final{\alpha})(\tau, \noprog{\tau})\), where $\mathsf{p}_0$ is the first projection from an ordered pair, is a non-wellfounded proof translation from \(\mathcal{C}_0\) to \(\mathcal{C}_1\).
    \end{enumerate}
\end{theorem}
\begin{proof}
    Proof of (i). Let \(\pi_0 \in \ffproof{\mathcal{C}_0}\) and \(\pi_1 = \final{\alpha}(\pi_0)\). Thanks to the second condition in the definition of proof translation step, by induction in \(r\) we can show that if \(\pi_0 \in \ffproof{\mathcal{C}_0}\) then \(\subelement{\alpha}{\pi_0}{r} \in \ffproof{\mathcal{C}_0}\) (1).

    Let \(r \in \rootsp{\pi_1}\), we have to show that \((\gfragment{}{\pi_1}{r},w \mapsto \sequent{\gfragment{}{\pi_1}{r:w}}{\epsilon})\) is a proof fragment in \(\mathcal{C}_1\).
    By Lemma \ref{lm:fundamental-properties} and thanks to \(\final{\alpha}\) being a coalgebra morphism, we know that \(\rootsp{\pi_1} = \rootsp{\alpha, \pi_0}\) and \(\gfragment{}{\pi_1}{r} = \gfragment{\alpha}{\pi_0}{r} = \gfragment{\alpha}{\subelement{\alpha}{\pi_0}{r}}{\nil}\). Given \(w \in \nwleaf{\gfragment{}{\pi_1}{r}}\), with Lemma \ref{lm:fundamental-properties} again, we get \(r{:}w \in \rootsp{\pi_1} = \rootsp{\alpha, \pi_0}\) and \(\gfragment{}{\pi_1}{r:w} = \gfragment{\alpha}{\pi_0}{r:w} = \gfragment{\alpha}{\subelement{\alpha}{\pi_0}{r}}{[w]}\). So it suffices to show that
    \(
    (\gfragment{\alpha}{\subelement{\alpha}{\pi_0}{r}}{\nil}, w \mapsto \sequent{\gfragment{\alpha}{\subelement{\alpha}{\pi_0}{r}}{[w]}}{\epsilon})
    \)
    is a proof fragment of \(\mathcal{C}_1\). But this is the first condition in the definition of proof translation step, which is obtained thanks to \(\subelement{\alpha}{\pi_0}{r} \in \ffproof{\mathcal{C}_0}\) by (1).

    Proof of (ii). Note that if \(\tau \in \illproof{\mathcal{C}_0}\) then \((\tau,\noprog{\tau})\in \ffproof{\mathcal{C}_0}\) by Lemma \ref{lm:equivalent-proof-notions} and thanks to point (i) of this theorem \(\pi_1 = \final{\alpha}(\tau,\sim^\tau) \in \ffproof{\mathcal{C}_1}\). Applying Lemma \ref{lm:equivalent-proof-notions} we conclude that \(\mathsf{p}_0(\pi_1) \in \illproof{\mathcal{C}_1}\).
\end{proof}

Let us finish this section remarking that this result can be generalized. We can redefine the notion of proof translation step as starting in an arbitrary set, say that \(\alpha : X \longrightarrow \mathcal{T}(X)\) \emph{is a proof translation step from \(X\) to the calculus \(\mathcal{C}\)} iff for any \(x \in X\), \((\fragment{\alpha}{\nil}(x), w \mapsto \sequent{\fragment{\alpha}{w}(x)}{\epsilon}\) is a proof fragment in \(\mathcal{C}\). Then it is easy to adapt the proof of Theorem \ref{th:proof-translation-step-to-proof-translation} to show that extending this kind of proof translation step via corecursion gives a function from \(X\) to proofs in \(\mathcal{C}\). One possible use of this is to define corecursions that change their character. For example, imagine you want to define a function from \(X\) to proofs in \(\mathcal{C}\), that first follows some proof translation step \(\alpha\) and then it follows another proof translation step \(\beta\). Then we can define a proof translation step \(\gamma : X \times \{0,1\} \longrightarrow \mathcal{C}\)\footnote{Note that in order to \(\gamma\) being a proof translation step from \(X\) to the calculus \(\mathcal{C}\) we also need a compatibility condition between \(\alpha\) and \(\beta\). The compatibility condition can be intuitively expressed as: if we find \((x,0)\) and it meets the condition for changing the behaviour from \(\alpha\) to \(\beta\) then \((\iota^\alpha_x, w \mapsto \sequent{\iota^\beta_{\mu^\alpha_x(w)}}{\epsilon})\) must be a proof fragment in \(\mathcal{C}\).} as:
\[
\gamma(x,i) = \begin{cases}
    (\iota^\alpha_x, (w \mapsto \mu^\alpha_x(w),0)) &\text{if \(i = 0\) and the condition is not true}, \\
    (\iota^\alpha_x, (w \mapsto \mu^\alpha_x(w),1)) &\text{if \(i = 0\) and the condition is true}, \\
    (\iota^\beta_x, (w \mapsto \mu^\beta_x(w),1)) &\text{if \(i = 1\)}.
    \end{cases}
\]
The desired translation of \(x \in X\) would be to apply the translation obtained by this proof translation step \(\gamma\) to \((x,0)\).
\section{Cut-Elimination for Non-wellfounded Grz}

In this section we illustrate the method of translations between local progress sequent systems developed in Section \ref{sec:proof-trans} by proving cut-elimination for a non-wellfounded calculus for $\mathsf{Grz}$. The calculus was introduced by Savateev and Shamkanov \cite{ShamkanovGrz}, who also establish cut-elimination. We briefly introduce the language of $\mathsf{Grz}$ and the aforementioned non-wellfounded calculus. We then show how instances of cut can be pushed out of the main fragment of a proof, and apply our method to obtain the cut-elimination result.

The language of $\mathsf{Grz}$ is the language of modal logic with primitive symbols \(\bot\), \(\to\), and \(\nec\). We will use \(p,q,r\) to denote propositional variables and \(\phi,\psi,\chi\) to denote formulas. \(\Gamma\), \(\Delta\), \(\Lambda\), \(\Pi\) will denote finite multisets of formulas. Sequents are ordered pairs of shape \((\Gamma,\Delta)\), usually denoted by \(S\). Union of sequents is the componentwise union of multisets and is denoted by \(S_0 \cup S_1\). We also define the notations:
\begin{align*}
    &S,\phi^\bullet\text{ to mean }S \cup (\{\phi\},\varnothing),
    && S,\phi^\circ\text{ to mean }S \cup (\varnothing,\{\phi\}), \\
    &S,\Gamma^\bullet\text{ to mean }S \cup (\Gamma,\varnothing),
    &&S,\Gamma^\circ\text{ to mean }S \cup (\varnothing,\Gamma).
\end{align*}
where \(\{\phi\}\) is the  singleton multiset consisting of $\phi$.

\begin{definition}
    We define the following rules:
    \begin{align*}
        \prfbyaxiom{Ax}{S,p^\bullet,p^\circ}
        & &
        \prfbyaxiom{$\bot^\bullet$}{S,\bot^\bullet} \\
        \prftree[r]{$\to^\bullet$}
        {S,\phi^\circ}{}
        {S,\psi^\bullet}
        {S,\phi \to \psi^\bullet}
        & &
        \prftree[r]{$\to^\circ$}
        {S,\phi^\bullet,\psi^\circ}
        {S,\phi \to \psi^\circ} \\
        \prftree[r]{Refl}
        {S,\phi^\bullet,\nec\phi^\bullet}
        {S,\nec\phi^\bullet}
        &&
        \prftree[r]{$\nec$}
        {S,\nec\Pi^\bullet,\phi^\circ}{}
        {\nec\Pi^\bullet,\phi^\circ}
        {S,\nec\Pi^\bullet,\nec\phi^\circ}
    \end{align*}
    In each rule the explicitly displayed formula in the conclusion is called \emph{principal}. In the rule $\nec$ the sequent $S$ is called the \emph{weakening part}.

    Then \(\Grz^\infty\) is the local progress sequent calculus with the rules above such that the only progress occurs in the right premise of the \(\nec\) rule. We also define the rule:
    \[
    \prftree[r]{\cut}
    {S,\phi^\circ}{}
    {S,\phi^\bullet}
    {S}
    \]
    The system consisting of the rules of \(\Grz^\infty\) plus cut and where the only progress occurs in the right premise of \(\nec\) is denoted as \((\Grz + \cut)^\infty\)
\end{definition}

Thanks to Section \ref{sec:proof-trans} we know that \(\tau\) is an non-wellfounded proof of \(\cal{C}\) iff \((\tau,\noprog{\tau})\) ($\noprog{\tau}$ is defined in Definition \ref{df:fragment-relation-in-proofs}) is a finite-fragmented proof of \(\cal{C}\). For that reason it will be common to identify both objects and for example talk about the main fragment of \(\tau\) when we are strictly speaking of the main fragment of \((\tau,\noprog{\tau})\). With this in mind, we define the notion of \emph{local height} of a proof \(\pi\) as the height of its main fragment.

\begin{definition}
    A function \(f : \mathbb{P}^\infty\Grzc \longrightarrow \mathbb{P}^\infty\Grzc\) is said to:
    \begin{enumerate}
        \item \emph{Preserve local height} iff the local height of \(f(\pi)\) is smaller or equal than the local height of \(\pi\).
        \item \emph{Preserve freeness of cuts in the main fragment} iff whenever \(\pi\) is a proof with no cuts in its main fragment, so is \(f(\pi)\).
    \end{enumerate}
\end{definition}

\begin{lemma}
    \label{lm:auxiliary-functions}
    We have functions \(\weakening_{S'}\), \(\contr_{p^\bullet}\), \(\contr_{p^\circ}\), \(\inv_{\bot^\circ}\), \(\linv_{\phi \to \psi^\bullet}\), \(\rinv_{\phi \to \psi^\bullet}\), \(\inv_{\phi \to \psi^\circ}\), \(\inv_{\nec\phi^\circ} : \mathbb{P}^\infty\Grzc \longrightarrow \mathbb{P}^\infty\Grzc\) such that:
    \begin{enumerate}
        \item \label{first-Grz} If \(\pi \vdash S\), then $\weakening_{S'}(\pi) \vdash S,S'$.
        \item If \(\pi \vdash S,p^\bullet,p^\bullet\), then \(\contr_{p^\bullet}(\pi) \vdash S,p^\bullet\).
        \item If \(\pi \vdash S,p^\circ,p^\circ\), then \(\contr_{p^\circ}(\pi) \vdash S,p^\circ\).
        \item If \(\pi \vdash S,\bot^\circ\), then \(\inv_{\bot^\circ}(\pi) \vdash S\).
        \item If \(\pi \vdash S,\phi \to \psi^\bullet\), then \(\linv_{\phi \to \psi^\bullet}(\pi) \vdash S,\phi^\circ\) and \(\rinv_{\phi \to \psi^\bullet}(\pi) \vdash S,\psi^\bullet\).
        \item If \(\pi \vdash S,\phi \to \psi^\circ\), then \(\inv_{\phi \to \psi^\circ}(\pi) \vdash S,\phi^\bullet,\psi^\circ\).
        \item \label{last-Grz} If \(\pi \vdash S,\nec \phi^\circ\), then \(\inv_{\nec\phi^\circ}(\pi) \vdash S,\phi^\circ\).
        \item \label{extra-Grz}  All these functions preserve local height and freeness of cuts in the main fragment. 
    \end{enumerate}
\end{lemma}
\begin{proof}
    All the items (\ref{first-Grz}) to (\ref{last-Grz}) are provable by induction in the local height of $\pi$, (\ref{extra-Grz}) is just an observation of the constructions during the induction.
\end{proof}

\begin{lemma}
    \label{lm:pusing-topmost-cut-upwards}
    Let \(\pi \vdash_{\Grzc} S,\phi^\circ\) and \(\tau \vdash_{\Grzc} S,\phi^\bullet\) without any instance of \(\cut\) in their main fragments. Then there is a proof of \(\vdash_{\Grzc} S\) without any instance of $\cut$  in its main fragment.
\end{lemma}
\begin{proof}
    To any such pair of proofs \((\pi,\tau)\) we can assign an ordinal \(n\omega + m\) where \(n\) is the rank of \(\phi\) (the cut formula) and \(m\) is the sum of the local heights of \(\pi\) and \(\tau\). We proceed by strong induction in this ordinal and by cases in the shape of \(\pi\) and \(\tau\) as follows.

    Case 1: either \(\pi\) or \(\tau\) consists of an initial sequent only. In this case it suffices to consider if \(S\) is still an initial sequent and in case it is not use \(\inv_{\bot^\circ}\), \(\contr_{p^\bullet}\) or \(\contr_{p^\circ}\) from Lemma \ref{lm:auxiliary-functions}.

    Case 2: the main formulas of \(\pi\) and \(\tau\) are the cut-formula. There are two subcases, if both formulas are implication it must be that \(\pi\) ends in an application of \(\to^\circ\) and \(\tau\) in an application of \(\to^\bullet\). Then, it is easy to see that the cut can be replaced using the I.H. in the premises since the rank of the cut-formula is smaller, applying \(\weakening\) from Lemma \ref{lm:auxiliary-functions} when needed. The other possibility is that the main formulas are box-formulas. We have that \(\pi\) and \(\tau\) have shape:
        \begin{align*}
            \prftree[r]{$\nec$}
            {\prftree[noline]{\pi_0}{S_0,\nec\Pi^\bullet,\phi_0^\circ}}
            {\prftree[noline]{\pi_1}{\nec\Pi^\bullet,\phi_0^\circ}}
            {S_0,\nec\Pi^\bullet,\nec\phi_0^\circ}
            &&
            \prftree[r]{Refl}
            {\prftree[noline]{\tau_0}{S_0,\nec\Pi^\bullet,\phi_0^\bullet,\nec\phi_0^\bullet}}
            {S_0,\nec\Pi^\bullet,\nec\phi_0^\bullet}
        \end{align*}
        where \(S = S_0,\nec\Pi^\bullet\) and \(\phi = \nec \phi_0\).
        By induction hypothesis applied to \(\weakening_{\phi^\bullet_0}(\pi),\tau_0\) (the local height of \(\tau_0\) is strictly smaller than the local height of \(\tau\) and the local height of the other proof is smaller or equal that the local height of \(\pi\) by Lemma \ref{lm:auxiliary-functions}) we get a proof \(\iota_0 \vdash S_0,\nec\Pi^\bullet,\phi^\bullet_0\) without cuts in its main fragment. We can apply the induction hypothesis to \(\pi_0,\iota_0\) (the rank of \(\phi_0\) is smaller than the rank of \(\nec\phi_0\)) and obtain \(\iota_1 \vdash S_0,\nec\Pi^\bullet\) with no cuts in its main fragment. This is the desired proof.

    Case 3: the principal formula of either \(\pi\) or \(\tau\) is not the cut-formula. We have five subcases, three of them correspond to the principal formula not being the cut-formula being the result of \(\to^\bullet,\to^\circ\) or \(\text{refl}\). In these cases we just need to use \(\linv\) and \(\rinv\),\(\inv\) or \(\weakening\) (from Lemma \ref{lm:auxiliary-functions}), apply the I.H. and then apply \(\to^\bullet, \to^\circ\) or \(\text{refl}\) respectively. We explicitely prove the other two subcases.

    \emph{The last rule of \(\tau\) (the case with \(\pi\) is analogous) is \(\nec\) and the cut-formula is in the weakening part of the rule instance}. \(\pi\) and \(\tau\) have shape:
        \begin{align*}
            \prftree[noline]{\pi}
            {S_0,\nec\Pi^\bullet,\nec\psi^\circ,\phi^\circ}
            &&
            \prftree[r]{$\nec$}
            {\prftree[noline]{\tau_0}{S_0,\nec\Pi^\bullet,\psi^\circ,\phi^\bullet}}
            {\prftree[noline]{\tau_1}{\nec\Pi^\bullet,\psi^\circ}}
            {S_0,\nec\Pi^\bullet,\nec\psi^\circ,\phi^\bullet}
        \end{align*}
        where \(S = S_0,\nec\Pi^\bullet,\nec\psi^\circ\). Let \(\iota_0 = \inv_{\nec\psi^\circ}(\pi) \vdash S_0,\nec\Pi^\bullet,\psi^\circ,\phi^\bullet\), we can apply the induction hypothesis to \(\iota_0,\tau_0\) (the local height is reduced thanks to Lemma \ref{lm:auxiliary-functions}) to obtain \(\iota_1 \vdash S_0,\nec\Pi^\bullet,\psi^\circ\) with no cuts in its main fragment. The desired proof is:
        \[
        \prftree[r]{$\nec$}
        {\prftree[noline]{\iota_1}{S_0,\nec\Pi^\bullet,\psi^\circ}}
        {\prftree[noline]{\tau_1}{\nec\Pi^\bullet,\psi^\circ}}
        {S_0,\nec\Pi^\bullet,\nec\psi^\circ}
        \]

    \emph{The last rule of \(\tau\) is \(\nec\) and the cut-formula is not the principal formula but it is not in the weakening part of the rule instance} (the case with \(\pi\) is impossible since in \(\pi\) the cut-formula must be on the right of the sequent, so if it is not the principal formula it appears in the weakening part). \(\pi\) and \(\tau\) have the following shape:
        \begin{align*}
            \prftree[r]{$\nec$}
            {\prftree[noline]{\pi_0}{S_0,\nec\Pi^\bullet,\nec\psi^\circ,\phi_0^\circ}}
            {}
            {\prftree[noline]{\pi_1}{\nec\Pi^\bullet,\phi_0^\circ}}
            {S_0,\nec\Pi^\bullet,\nec\psi^\circ,\nec\phi_0^\circ}
            &&
            \prftree[r]{$\nec$}
            {\prftree[noline]{\tau_0}{S_1,\nec\Lambda^\bullet,\psi^\circ,\nec\phi_0^\bullet}}
            {}
            {\prftree[noline]{\tau_1}{\nec\Lambda^\bullet,\psi^\circ,\nec\phi_0^\bullet}}
            {S_1,\nec\Lambda^\bullet,\nec\psi^\circ,\nec\phi_0^\bullet}
        \end{align*}
        where \(S = S_0,\nec\Pi^\bullet,\nec\psi^\circ = S_1,\nec\Lambda^\bullet,\nec \psi^\circ\) (so $S_0,\nec\Pi^\bullet = S_1,\nec\Lambda^\bullet$ (i)) and \(\phi = \nec\phi_0\). Note that \(\Pi \cup (\Lambda \setminus \Pi) = \Lambda \cup (\Pi \setminus \Lambda)\) and define \(S_2 = S_0 \cap S_1\), so \(S_0 = S_2,\nec(\Lambda \setminus \Pi)^\bullet\) and \(S_1 = S_2,\nec(\Pi \setminus\Lambda)^\bullet\). Let \(\pi' = \inv_{\nec\psi^\circ}(\pi) \vdash S_0,\nec\Pi^\bullet,\psi^\circ,\nec\phi^\circ_0\). Due to (i) we can apply the induction hypothesis to \(\pi',\tau_0\) (the local height is reduced and there are no cuts in the main fragments thanks to Lemma \ref{lm:auxiliary-functions} and because we are in the main fragment of \(\tau\) still) and obtain a proof \(\iota_0 \vdash S_0,\nec\Pi^\bullet,\psi^\circ = S_2,\nec\Pi^\bullet,\nec(\Lambda\setminus \Pi)^\bullet,\psi^\circ\) with no cuts in its main fragment. Let \(\rho\) be the proof:
        \[
        \prftree[r]{\cut}
            {
                \prftree[r]{$\nec$}
                {
                    \prftree[noline]{\weakening_{\nec(\Lambda \setminus \Pi)^\bullet,\psi^\circ}(\pi_1)}
                    {\nec\Pi^\bullet,\nec(\Lambda\setminus\Pi)^\bullet,\psi^\circ,\phi_0^\circ}
                }
                {
                    \prftree[noline]{\pi_1}
                    {\nec\Pi^\bullet,\phi_0^\circ}
                }
                {
                    \nec\Pi^\bullet,\nec(\Lambda\setminus\Pi)^\bullet,\psi^\circ,\nec\phi_0^\circ
                }
            }
            {}
            {
                \prftree[noline]{\weakening_{\nec(\Pi\setminus \Lambda)^\bullet}(\tau_1)}{\nec\Pi^\bullet,\nec(\Lambda\setminus\Pi)^\bullet,\psi^\circ,\nec\phi_0^\bullet}
            }
            {\nec\Pi^\bullet,\nec(\Lambda\setminus\Pi)^\bullet,\psi^\circ}
        \]
        Then, the desired proof is   
        \[
        \prftree[r]{$\nec$}
        {\prftree[noline]{\iota_0}{S_2,\nec\Pi^\bullet,\nec(\Lambda\setminus \Pi)^\bullet,\psi^\circ}}
        {}
        {\prftree[noline]{\rho}{\nec \Pi^\bullet, \nec (\Lambda \setminus \Pi)^\bullet, \psi^\circ}}
        {S_2,\nec\Pi^\bullet,\nec(\Lambda \setminus \Pi)^\bullet, \nec\psi^\circ}
        \]
        We notice that the constructed proof does not have any cuts in its main fragment, since the right premise of \(\nec\) is outside the main fragment of the proof.
\end{proof}

\begin{lemma}
    \label{lm:pushing-cuts-upwards}
    There is a function \(\mathsf{cuts}\textsf{-}\mathsf{up} : \mathbb{P}^\infty\Grzc \longrightarrow \mathbb{P}^\infty\Grzc\) such that if \(\pi \vdash S\) then \(\mathsf{cuts}\textsf{-}\mathsf{up}(\pi) \vdash S\) and \(\mathsf{cuts}\textsf{-}\mathsf{up}(\pi)\) has no instances of \(\cut\) in its main fragment.
\end{lemma}
\begin{proof}
    This is a simple induction in the number of cuts in the main fragment of \(\pi\). For the inductive step just go to a top-most cut in the main fragment (i.e. a cut without any cuts that are above it in the main fragment) and apply Lemma \ref{lm:pusing-topmost-cut-upwards} to get rid of that cut.
\end{proof}

Thanks to the tools of Section \ref{sec:proof-trans} the cut-elimination proof from Lemma \ref{lm:pushing-cuts-upwards} is straightforward.

\begin{theorem}
    \label{th:cut-elim-grzc}
    If \(\vdash_{\Grzc} S\), then \(\vdash_{\Grz^\infty} S\).
\end{theorem}
\begin{proof}
    Define the function \(\alpha : \mathbb{P}^{\text{ff}}\Grzc \longrightarrow \treefun(\mathbb{P}^\text{ff})\) that, given \mbox{\(\pi = (\tau,\sim)\)}, outputs \((\iota,\mu) = \destruct(\tau',\sim^{\tau'})\) where \(\tau' = \mathsf{cuts}\text{-}\mathsf{up}(\tau)\). Since \(\iota\) is the main fragment of \(\mathsf{cuts}\textsf{-}\mathsf{up}(\tau)\), it is a proof-fragment in \(\Grzc\).  Since there are no instances of \(\cut\) in $\iota$ by Lemma \ref{lm:pushing-cuts-upwards}, and \(\Grzc\) only progresses in the same rule instances as in \(\Grz^\infty\), it is clear that $\iota$ is a proof fragment of \(\Grz^\infty\). Since \(\mu(w)\) is by definition a subtree of \(\mathsf{cuts}\textsf{-}\mathsf{up}\) which is a proof in \(\Grzc\) we have that \(\mu(w)\) is also a proof in \(\Grzc\). We can conclude that the conditions to apply Theorem \ref{th:proof-translation-step-to-proof-translation} are fulffiled. Hence we obtain a proof translation from \(\Grzc\) to \(\Grz^\infty\) that does not change the conclusion sequence (by definition of the \(\alpha\)) so we have the desired result.
\end{proof}

We conclude that in any local progress sequent calculi where cuts can be pushed outside the main fragment, we can show cut-elimination. In particular, we remark that cut-elimination can be proven for weak Grzegorczyk modal logic $\mathsf{wGrz}$ (see \cite{ShamkanovwGrz} to see a non-wellfounded proof system for this logic and a proof of its cut-elimination using ultrametric spaces) by adapting the function $\mathsf{cuts}\textsf{-}\mathsf{up}$ accordingly. This implies that cut-elimination for local progress sequent calculi is quite close to cut-elimination of finitary systems since the main fragment is always finite. 

\section{Final remarks and future work}

We have introduced a uniform method to define proof translations for non-wellfounded proofs with a local progress condition. Our method is based on the observation that it suffices to translate a finite fragment of a proof and then define from such a translation step a map by corecursion. The translation method is uniform in the sense that the specific proof systems, between which a translation shall be defined, as well as the specific type of translation, does not matter. We have illustrated the method by proving cut-elimination for a non-wellfounded calculus for $\mathsf{Grz}$. 

A natural direction for further research is to study the relationship between our corecursive method for cut-elimination and the ultrametric approach by Shamkanov and Savateev. At the time of writing this paper it is unknown to us how exactly these two method relate and whether using ultrametric spaces is strictly more powerful than our approach. Another direction that we are currently pursuing is to explore which logics can be given a complete non-wellfounded local progress calculus. The standard examples are $\mathsf{GL}$ and $\mathsf{Grz}$, however we suspect that many modal fixed point logics belonging to the alternation free fragment of the modal mu-calculus may also be given a local progress calculus. For example modal logic with the master modality admits such a calculus\footnote{Adding ordinal annotations, which may complicate the use of this method.}, leading to the question whether we can establish cut-elimination for this calculus using our corecursive method. Finally, another pressing open question is whether a similar method for proof translations can be defined for non-wellfounded calculi based on a trace condition. We expect this to be a difficult problem, since it is generally challenging to ensure that traces are preserved when translating proofs.

\bibliographystyle{aiml}
\bibliography{aiml}

\begin{thebibliography}{10}
\expandafter\ifx\csname url\endcsname\relax
  \def\url#1{\texttt{#1}}\fi
\expandafter\ifx\csname urlprefix\endcsname\relax\def\urlprefix{URL }\fi
\newcommand{\enquote}[1]{``#1''}

\bibitem{Afshari2023}
Afshari, B., L.~Grotenhuis, G.~E. Leigh and L.~Zenger, \emph{Ill-founded proof systems for intuitionistic linear-time temporal logic}, in: R.~Ramanayake and J.~Urban, editors, \emph{Automated Reasoning with Analytic Tableaux and Related Methods} (2023), pp. 223--241.

\bibitem{cut-cyclic}
Afshari, B. and J.~Kloibhofer, \emph{Cut elimination for cyclic proofs: A case study in temporal logic}, in: \emph{Proceedings Twelfth International Workshop on Fixed Points in Computer Science} (to appear).

\bibitem{turata2023}
Afshari, B., G.~E. Leigh and G.~Men{\'e}ndez~Turata, \emph{A cyclic proof system for full computation tree logic}, in: \emph{31st EACSL Annual Conference on Computer Science Logic (CSL 2023)}, Schloss Dagstuhl-Leibniz-Zentrum f{\"u}r Informatik, 2023.

\bibitem{baelde2016}
Baelde, D., A.~Doumane and A.~Saurin, \emph{{Infinitary Proof Theory: the Multiplicative Additive Case}}, in: J.-M. Talbot and L.~Regnier, editors, \emph{25th EACSL Annual Conference on Computer Science Logic (CSL 2016)},  Leibniz International Proceedings in Informatics (LIPIcs)  \textbf{62} (2016), pp. 42:1--42:17.
\newline\urlprefix\url{https://drops-dev.dagstuhl.de/entities/document/10.4230/LIPIcs.CSL.2016.42}

\bibitem{Brotherston-Simpson:07}
Brotherston, J. and A.~Simpson, \emph{Complete sequent calculi for induction and infinite descent}, in: \emph{Proceedings of {LICS}-22} (2007), pp. 51--60.

\bibitem{BUCHELI201083}
Bucheli, S., R.~Kuznets and T.~Studer, \emph{Two ways to common knowledge}, Electronic Notes in Theoretical Computer Science \textbf{262} (2010), pp.~83--98, proceedings of the 6th Workshop on Methods for Modalities (M4M-6 2009).
\newline\urlprefix\url{https://www.sciencedirect.com/science/article/pii/S1571066110000290}

\bibitem{Das2018}
Das, A. and D.~Pous, \emph{{Non-Wellfounded Proof Theory For (Kleene+Action)(Algebras+ Lattices)}}, in: D.~R. Ghica and A.~Jung, editors, \emph{27th EACSL Annual Conference on Computer Science Logic (CSL 2018)},  Leibniz International Proceedings in Informatics (LIPIcs)  \textbf{119} (2018), pp. 19:1--19:18.
\newline\urlprefix\url{https://drops-dev.dagstuhl.de/entities/document/10.4230/LIPIcs.CSL.2018.19}

\bibitem{Docherty2019}
Docherty, S. and R.~N.~S. Rowe, \emph{A non-wellfounded, labelled proof system for propositional dynamic logic}, in: \emph{Automated Reasoning with Analytic Tableaux and Related Methods: 28th International Conference, TABLEAUX 2019, London, UK, September 3-5, 2019, Proceedings} (2019), p. 335–352.
\newline\urlprefix\url{https://doi.org/10.1007/978-3-030-29026-9_19}

\bibitem{KokkinisStuder+2016+171+192}
Kokkinis, I. and T.~Studer, \emph{Cyclic proofs for linear temporal logic}, in: D.~Probst and P.~Schuster, editors, \emph{Concepts of Proof in Mathematics, Philosophy, and Computer Science}, De Gruyter, Berlin, Boston, 2016 pp. 171--192.
\newline\urlprefix\url{https://doi.org/10.1515/9781501502620-011}

\bibitem{Marti2021}
Marti, J. and Y.~Venema, \emph{A focus system for the alternation-free $\mu$-calculus}, in: \emph{Automated Reasoning with Analytic Tableaux and Related Methods: 30th International Conference, TABLEAUX 2021, Birmingham, UK, September 6–9, 2021, Proceedings} (2021), p. 371–388.
\newline\urlprefix\url{https://doi.org/10.1007/978-3-030-86059-2_22}

\bibitem{Niwinski1996}
Niwiński, D. and I.~Walukiewicz, \emph{Games for the $\mu$-calculus}, Theoretical Computer Science \textbf{163} (1996), pp.~99--116.
\newline\urlprefix\url{https://www.sciencedirect.com/science/article/pii/0304397595001360}

\bibitem{Rooduijn22}
Rooduijn, J. M.~W. and L.~Zenger, \emph{An analytic proof system for common knowledge logic over s5}, in: \emph{David Fern{\'{a}}ndez-Duque, Alessandra Palmigiano and Sophie Pinchinat (eds.) Advances in Modal Logic}, 2022, pp. 659--680.

\bibitem{Saurin2023}
Saurin, A., \emph{A linear perspective on cut-elimination for non-wellfounded sequent calculi with least and greatest fixed-points}, in: R.~Ramanayake and J.~Urban, editors, \emph{Automated Reasoning with Analytic Tableaux and Related Methods} (2023), pp. 203--222.

\bibitem{ShamkanovGrz}
Savateev, Y. and D.~Shamkanov, \emph{Non-well-founded proofs for the {G}rzegorczyk modal logic}, The Review of Symbolic Logic \textbf{14} (2018).

\bibitem{ShamkanovwGrz}
Savateev, Y. and D.~Shamkanov, \emph{Cut elimination for the weak modal grzegorczyk logic via non-well-founded proofs}, in: R.~Iemhoff, M.~Moortgat and R.~de~Queiroz, editors, \emph{Logic, Language, Information, and Computation} (2019), pp. 569--583.

\bibitem{ShamkanovGl}
Shamkanov, D.~S., \emph{Circular proofs for the {G}ödel-{L}öb provability logic}, Mathematical Notes \textbf{96} (2014), pp.~575--585.
\newline\urlprefix\url{http://dx.doi.org/10.1134/S0001434614090326}

\bibitem{Simpson2017}
Simpson, A., \emph{Cyclic arithmetic is equivalent to peano arithmetic}, in: J.~Esparza and A.~S. Murawski, editors, \emph{Foundations of Software Science and Computation Structures} (2017), pp. 283--300.

\bibitem{Stirling2013}
Stirling, C., \emph{A proof system with names for modal mu-calculus}, Electronic Proceedings in Theoretical Computer Science \textbf{129} (2013).

\bibitem{Studer2008}
Studer, T., \emph{On the proof theory of the modal mu-calculus}, Studia Logica \textbf{89} (2008), pp.~343--363.
\newline\urlprefix\url{http://www.jstor.org/stable/40268983}

\bibitem{guillermo}
Turata, G.~M., \emph{Cyclic proof systems for modal fixpoint logics}, ILLC Dissertation series  (2024).

\bibitem{VENEMA2007331}
Venema, Y., \emph{Algebras and coalgebras}, in: P.~Blackburn, J.~{Van Benthem} and F.~Wolter, editors, \emph{Handbook of Modal Logic},  Studies in Logic and Practical Reasoning  \textbf{3}, Elsevier, 2007 pp. 331--426.
\newline\urlprefix\url{https://www.sciencedirect.com/science/article/pii/S1570246407800097}

\end{thebibliography}

\end{document}